\documentclass[reqno, 11pt]{amsart}

\makeatletter
\let\origsection=\section \def\section{\@ifstar{\origsection*}{\mysection}} 
\def\mysection{\@startsection{section}{1}\z@{.7\linespacing\@plus\linespacing}{.5\linespacing}{\normalfont\scshape\centering\S}}
\makeatother    

\usepackage{geometry}
\numberwithin{equation}{section}
\numberwithin{figure}{section}

\usepackage{xcolor}
\usepackage{hyperref}
\hypersetup{
    unicode,
    colorlinks,
    linkcolor={red!60!black},
    citecolor={green!60!black},
    urlcolor={blue!60!black}
}

\usepackage{amsfonts,amsthm,amssymb,amsmath}
\usepackage[T1]{fontenc}
\usepackage{microtype}
\usepackage{graphicx}
\usepackage{tikz}
\usepackage{tkz-berge}
\usepackage[usenames,dvipsnames]{pstricks}
\usepackage{epsfig}
\usepackage{verbatim}
\usepackage{scalerel}

\usepackage{enumerate,enumitem}
\usepackage{thmtools, thm-restate}
\usetikzlibrary{decorations.pathmorphing, arrows, calc, matrix, arrows.meta}

\usepackage{mathtools}
\usepackage{subcaption}

\newtheorem{theorem}{Theorem}
\numberwithin{theorem}{section}
\newtheorem{lemma}[theorem]{Lemma}  
\newtheorem{cor}[theorem]{Corollary}

\theoremstyle{definition}

\newtheorem{prob}{Problem}

\newcommand{\script}{\mathcal}

\newcommand{\Set}[1]{{\left\lbrace {#1} \right\rbrace}}

\newcommand{\cardinality}[1]{{\left\lvert {#1} \right\rvert}}

\def\set#1:#2{\Set{{#1} \colon {#2}}}
\def\downcl#1{\lceil{#1}\rceil}
\def\upcl#1{\lfloor{#1}\rfloor}

\newcommand{\N}{\mathbb{N}}
\newcommand{\sub}{\subseteq}

\renewcommand{\triangleleft}{\vartriangleleft}
\renewcommand{\leq}{\leqslant}
\renewcommand{\geq}{\geqslant}
\renewcommand{\preceq}{\preccurlyeq}
\renewcommand{\rho}{\varrho}
\renewcommand{\subset}{\subseteq}

\newcommand{\nottriangleleft}{\not\kern-1pt\mathrel{\triangleleft}}

\DeclareMathOperator{\cf}{cf}

\begin{document}
\title{A new obstruction for normal spanning trees}

\author{Max Pitz}
\address{Hamburg University, Department of Mathematics, Bundesstra\ss e 55 (Geomatikum), 20146 Hamburg, Germany}
\email{max.pitz@uni-hamburg.de}



\begin{abstract}
In a paper from 2001 (Journal of the LMS), Diestel and Leader offered a proof that a connected graph has a normal spanning tree if and only if it does not contain a minor from two specific forbidden classes of graphs, all of cardinality $\aleph_1$.

Unfortunately, their proof contains a gap, and their result is incorrect. In this paper, we construct a third type of obstruction: an $\aleph_1$-sized graph without a normal spanning tree that contains neither of the two types described by Diestel and Leader as a minor. Further, we show that any list of forbidden minors characterising the graphs with normal spanning trees must contain graphs of arbitrarily large cardinality.
\end{abstract}

\maketitle

\section{Introduction}

A rooted spanning tree $T$ of a graph $G$ is called \emph{normal} if the ends of any edge of $G$ are comparable in the natural tree order of $T$. Intuitively, the edges of $G$ run `parallel' to branches of $T$, but never `across'. All countable connected graphs have normal spanning trees, but uncountable graphs might not, as demonstrated by complete graphs on uncountably many vertices.

Halin observed in \cite{halin2000miscellaneous} 
that the property of having a normal spanning tree is minor-closed, i.e.\ preserved under taking connected minors. Recall that a graph $H$ is a \emph{minor} of another graph $G$, written $H \preceq G$, if to every vertex $x \in H$ we can assign a (possibly infinite) connected set $V_x \subset V(G)$, called the \emph{branch set} of $x$, so that these sets $V_x$ are disjoint for different $x$ and $G$ contains a $V_x-V_y$ edge whenever $xy$ is an edge of $H$. 

In \cite[Problem~7.3]{halin2000miscellaneous} Halin asked for a forbidden minor characterisation for the property of having a normal spanning tree. In the universe of finite graphs, the famous Seymour-Robertson Theorem asserts that any minor-closed property of finite graphs can be characterised by \emph{finitely} many forbidden minors, see e.g.\ \cite[\S12.7]{Bible}. 
Whilst for infinite graphs, we generally need an infinite list of forbidden minors, Diestel and Leader \cite{DiestelLeaderNST} published a proof claiming that for the property of having a normal spanning tree, the forbidden minors come in two structural types: 

First, the class of $(\aleph_0,\aleph_1)$\emph{-graphs}, bipartite graphs $(A,B)$ such that $\cardinality{A}=\aleph_0$, $\cardinality{B}=\aleph_1$, and every vertex in $B$ has infinite degree. Structural results on this graph class can be found in \cite{BowlerGeschkePitzNST}.
And second, the class of \emph{Aronszajn-tree graphs}, graphs whose vertex set is an order theoretic Aronszajn tree $\script{T}$ (an order tree $(\script{T},\leq)$ of size $\aleph_1$ in which all levels and branches are countable) such that the down-neighbourhood of any node $t \in \script{T}$ is cofinal below~$t$.

However, there is a gap in Diestel and Leader's proof, and it turns out that their list of forbidden minors is incomplete:
 In Section~\ref{sec_counterexamples}, we exhibit a third obstruction for normal spanning trees -- a graph without normal spanning tree containing neither an $(\aleph_0,\aleph_1)$-graph nor an Aronszajn-tree graph as a minor. More significantly, we will see in Section~\ref{sec_counterexamples2} why any list of forbidden minors that works just under the usual axioms of set theory ZFC must contain graphs of arbitrary large cardinality. In between, in Section~\ref{sec_discussion}, we discuss how these new obstructions occur naturally when trying to build a normal spanning tree.

Where does this leave us? Fortunately, all is not lost. Indeed, this new third obstruction only demonstrates that the 3-way interaction between normal spanning trees, graphs on order trees and the colouring number of infinite graphs is deeper and more intriguing than initially thought. Recall that a graph $G$ has \emph{countable colouring number} if there is a well-order $\leq^*$ on $V(G)$ such that every vertex of $G$ has only finitely many neighbours preceding it in $\leq^*$. 
Every graph with a normal spanning tree, and hence every minor of it, has countable colouring number, as witnessed by well-ordering the graph level by level.  

The most important consequence of Diestel and Leader's proposed forbidden minor characterisation was that it would have implied Halin's conjecture \cite[Conjecture~7.6]{halin2000miscellaneous}, that a connected graph has a normal spanning tree if and only if every minor of it has countable colouring number.
In a paper in preparation \cite{pitz2020d}, I will give a direct proof of Halin's conjecture. From this, a revised forbidden minor characterisation of graphs with normal spanning trees can be deduced.

\section{Preliminaries}

We follow the notation in \cite{Bible}. Given a subgraph $H \subseteq G$, write $N(H)$ for the set of vertices in $G -H$ with a neighbour in $H$. 

\subsection{Normal spanning trees}

If $T$ is a (graph-theoretic) tree with root~$r$, we write $x \le y$
for vertices $x,y\in T$ if $x$ lies on the unique $r$--$y$ path in~$T$.
A rooted spanning tree $T \subset G$ is \emph{normal} if the ends of any edge of $G$ are comparable in this tree order on $T$.

A set of vertices $U \subset V(G)$ is \emph{dispersed} (in G) if every ray in $G$ can be separated from $U$ by a finite set of vertices. 
The following theorem  of Jung from  \cite{jung1969wurzelbaume}, from which we will use the implication  $(1) \Rightarrow (2)$ further below, characterises graphs with normal spanning trees, see also \cite{pitz2020} for a short proof.

\begin{theorem}[Jung]
\label{thm_Jung} The  following are equivalent  for a connected graph $G$:
\begin{enumerate}
\item $G$ has a normal spanning tree,
\item $G$ has a normal spanning tree for every choice of $r \in V(G)$  as the root, and
\item $V(G)$ is a countable union of dispersed sets.
\end{enumerate}
\end{theorem}

\subsection{Normal tree orders and \texorpdfstring{$T$}{T}-graphs}

A partially ordered set $(T,\le)$ is called an \emph{order tree} if it has a unique minimal element (called the \emph{root}) and all subsets of the form $\lceil t \rceil = \lceil t \rceil_T := \set{t' \in T}:{t'\le t}$
are well-ordered. Our earlier partial ordering on the vertex set of
a rooted graph-theoretic tree is an order tree in this sense.

Let $T$ be an order tree. A~maximal chain in~$T$ is called a \emph{branch}
of~$T$; note that every branch inherits a well-ordering from~$T$. The
\emph{height} of~$T$ is the supremum of the order types of its branches. The
\emph{height} of a point $t\in T$ is the order type of~$\mathring{\lceil t \rceil}  :=
\lceil t \rceil  \setminus \{t\}$. The set $T^i$ of all points at height $i$ is
the $i$th \emph{level} of~$T$, and we
write $T^{<i} := \bigcup\set{T^j}:{j < i}$.

The intuitive interpretation of a tree order as expressing height will also
be used informally. For example, we may say that $t$ is \emph{above}~$t'$
if $t > t'$, call $\lceil X \rceil = \lceil X \rceil _T := \bigcup \set{\lceil x \rceil}:{x\in
X}$ the \emph{down-closure} of~$X\sub T$. And we say that $X$ is \emph{down-closed}, or $X$ is a \emph{rooted subtree}, if $X=\lceil X \rceil $.

An order tree $T$ is \emph{normal} in a graph $G$, if $V(G) = T$
and the two ends of any edge of $G$ are comparable in~$T$. We call $G$ a
\emph{$T$-graph} if $T$ is normal in $G$ and the set of lower neighbours of
any point $t$ is cofinal in $\mathring{\lceil t \rceil}$. 
For later use recall down the following standard results about $T$-graphs, and refer the reader to \cite[\S2]{brochet1994normal} for details.

\begin{lemma}
\label{lem_Tgraphproperties}
Let $(T,\leq)$ be an order tree and $G$ a $T$-graph.
\begin{enumerate}
\item \label{itemT1} For incomparable vertices $t,t'$ in $T$, the set $\downcl{t} \cap \downcl{t'}$ separates $t$ from $t'$ in $G$.
\item \label{itemT2} Every connected subgraph of $G$ has a unique $T$-minimal element.
\item \label{itemT4} Every subgraph of $G$ induced by an up-set $\upcl{t}$ is  connected.
\item \label{itemT3} If $T' \subset T$ is down-closed, the components of $G - T'$ are spanned by the sets $\upcl{t}$ for $t$ minimal in $T-T'$.
\end{enumerate}
\end{lemma}

\subsection{Stationary sets and Fodor's lemma}
We denote ordinals by $i,j,k,\ell$, and identify $i = \set{j}:{j < i}$.
Let $\ell$ be any limit ordinal. A subset $A \subset \ell$ is \emph{unbounded} if $\sup A = \ell$, and \emph{closed} if $\sup (A \cap m) = m$ implies $m \in A$ for all limits $m < \ell$. The set $A$ is a \emph{club-set} in $\ell$ if it is both closed and unbounded.
A subset $S \subset \ell$ is \emph{stationary} (in $\ell$) if $S$ meets every club-set of $\ell$. For the following standard results about stationary sets see e.g.\ \cite[\S III.6]{Kunen}.

\begin{lemma}
\label{lem_stationary}
(1) If $\kappa$ is a regular uncountable cardinal, $S \subset \kappa$ is stationary and $S = \bigcup \set{S_n}:{i \in \N}$, then some $S_n$ is stationary.

(2) \emph{[Fodor's lemma]} If $\kappa$ is a regular uncountable cardinal, $S \subset \kappa$ stationary and $f \colon S \to \kappa$ is  such that $f(s)<s$ for all $s \in S$, then there is $i< \kappa$ such that $f^{-1}(i)$ is stationary. 
\end{lemma}

\section{A new obstruction for normal spanning trees of size \texorpdfstring{$\aleph_1$}{aleph1}}
\label{sec_counterexamples}

In this section we encounter a third obstruction for normal spanning trees -- a graph without normal spanning tree, but also without $(\aleph_0,\aleph_1)$-graph and Aronszajn-tree graph as a minor.

Consider the $\omega_1$-regular tree with tops, i.e.\ the order tree $(T,\leq)$ where the nodes of $T$ are all sequences of elements of $\omega_1$ of length $\leq\omega$, including the empty sequence. Let $t \leq t'$ if $t$ is a proper initial segment of $t'$, and let $G_{\omega_1}$ be any $T$-graph. 

Given a set $S \subset \omega_1$ of limit ordinals, choose for each $s \in S$ a cofinal (not necessarily increasing) sequence $f_s \colon \N \to s$, and let $F=F(S) = \set{f_s}:{s \in S}$ be the corresponding collection of sequences in $\omega_1$. Let $T(S)$ denote the subtree of $T$ given by all finite sequences in $T$ together with $F(S)$, and let $G_{\omega_1}(S)$ denote the corresponding induced subgraph of $G_{\omega_1}$. 

To our knowledge, such a collection $F(S) = \set{f_s}:{s \in S}$ of tree branches was first considered by Stone in \cite[\S5]{stone1963sigma} where it is shown that $F$ is not Borel in the end space of $G_{\omega_1}[T^{{<}\omega}]$.

\begin{theorem}
\label{thm_counter}
Let $S \subset \omega_1$ be stationary. Then $G_{\omega_1}(S)$ does not have a normal spanning tree, despite the fact that it contains neither an $(\aleph_0,\aleph_1)$-graph nor an Aronszajn-tree graph as a minor.
\end{theorem}

The three assertions of Theorem~\ref{thm_counter} are divided into Lemmas~\ref{lem_stationarytops}, \ref{lem_aleph0aleph_1} and \ref{lem_noAron} below.

\begin{lemma}
\label{lem_stationarytops}
Let $S \subset \omega_1$ be stationary. Then $G_{\omega_1}(S)$ does not have a normal spanning tree.
\end{lemma}

\begin{proof}
Suppose for a contradiction that $G_{\omega_1}(S)$ has a normal spanning tree $R$. By Lemma~\ref{lem_stationary}(1), for some level $n \in \N$ the set $S' = \set{s \in S}:{f_s \in R^n}$ is stationary. By  Lemma~\ref{lem_Tgraphproperties}(\ref{itemT1}), any two vertices $f_s \neq f_{s'}$ in $F(S')$ can be 
separated by $\downcl{f_s}_R \cap \downcl{f_{s'}}_R$, a set of at most $n$ vertices. 

However, by Lemma~\ref{lem_stationary}(1), there is a stationary subset $S'' \subset S'$ and $m \in \N$ such that the first $n+1$ neighbours of $f_s$ for $s \in S''$ are contained on exactly the same levels  of $T^{\leq m}$. After applying Fodor's Lemma~\ref{lem_stationary}(2) iteratively $m$ times, we get a stationary subset $S''' \subset S''$ such that $f_s(i) = f_{s'}(i) $ for all $s,s' \in S'''$ and $i \leq m$. So the vertices in $F(S''')$ have at least $n+1$ common neighbours, a contradiction.
 \end{proof}

\begin{lemma}
\label{lem_aleph0aleph_1}
If an $(\aleph_0,\aleph_1)$-graph $H$ is a minor of some $T$-graph $G$ with $\operatorname{height}(T) = \omega + 1$, then some countable subtree of $T^{{<}\omega}$ has uncountably many tops.
\end{lemma}

\begin{proof}
Let $H = (A,B)$ be an $(\aleph_0,\aleph_1)$-graph with $H \preceq G$. For $v \in V(H)$ write $t_v$ for the unique minimal node of the branch set of $v$ in $\mathcal{T}$, which exists by Lemma~\ref{lem_Tgraphproperties}(\ref{itemT2}). 
Observe that if $ab \in E(H)$, then $t_a$ and $t_b$ are comparable in $T$. 
Since $B$ is uncountable, 
there is a first level $\alpha \leq \omega$ such that $B' = \set{b \in B}:{t_b \in T^{\alpha}}$ is uncountable. 

Case 1: If $\alpha  = n < \omega$ is finite. Since every $b \in B'$ has infinite degree, there must be infinitely many $t_a$ comparable to such an $t_b \in T^n$. As $\lceil t_b \rceil$ is finite, it follows that every such $t_b$ for $b \in B'$ has some $t_a$ above it. But then all these $t_a$ are distinct, contradicting that $A$ is countable.

Case 2:  If $\alpha  = \omega$. Then the branch sets of $b \in B'$ are $\Set{t_b} \subset T^{\omega}$. And $\set{t_a}:{a \in N(b)}$ forms an infinite, and hence cofinal chain below $t_b$ for any $b \in B'$. Since $A$ is countable, $\lceil A \cap T^{<\omega} \rceil$ is a countable subtree of $T^{<\omega}$ that picks up uncountable many tops $t_b$ for $b \in B'$, a contradiction.
\end{proof}

To see that Lemma~\ref{lem_aleph0aleph_1} applies to $G_{\omega_1}(S)$, note that any countably subtree of $T$ contains only sequences with values in $i$ for some $i < \omega_1$. But then no $f_s$ for $s > i$ is a top of that subtree: since the sequence $f_s$ is cofinal in $s$, we have $f_s(n) \geq  i$ for at least one $n \in \N$.

\begin{lemma}
\label{lem_noAron}
If an Aronszajn tree graph $H$ is a minor of some $T$-graph $G$, then $\operatorname{height}(T) \geq \omega_1$.
\end{lemma}

\begin{proof}
Let $\mathcal{T}$ be an Aronszajn tree and suppose that a $\mathcal{T}$-graph $H$ embeds into some $T$-graph $G$ as a minor. Using the notation of the previous proof, if $T$ has countable height, there is a level $T^\alpha$ of $T$ such that $Y = \set{v \in V(H)}:{t_v \in T^{\alpha}}$ is uncountable. If we choose $\alpha$ minimal, then deleting the countable set $X = \set{v \in V(H)}:{t_v \in T^{<\alpha}}$ in $H$ separates all vertices of $Y$ from each other.

However, we have $X \subset \mathcal{T}^{<\beta}$ for some $\beta<\omega_1$. By the Aronszajn property, both $\mathcal{T}^{<\beta}$ and $\mathcal{T}^{\beta}$ are countable. Since $H$ is a $\mathcal{T}$-graph, all but countably many vertices of $H - X$ are contained in a connected subgraph (Lemma~\ref{lem_Tgraphproperties}(\ref{itemT4})) of the form $\lfloor t \rfloor$ for some $t \in \mathcal{T}^{\beta}$. This contradicts that $Y$ is uncountable.
\end{proof}

\section{Discussion of Diestel \& Leader's proof}
\label{sec_discussion}

In this section we discuss the gap in Diestel and Leader's proof, and see how clubs and stationary sets of $\omega_1$ appear naturally when taking Diestel and  Leader's proof strategy to its logical conclusion.

Very briefly, given the task of constructing a normal spanning tree for some connected graph $G$, Diestel and Leader aim to partition $G$ into countable  subgraphs $H_1,H_2,H_3,\ldots$ such that each component of $H_{n}$ has only finitely many neighbours in $\overline{H}_n = \bigcup_{i < n} H_n$, and these finitely many neighbours span a clique. Given this setup, using Jung's Theorem~\ref{thm_Jung}(2), it is then not hard to extend a normal spanning tree of $\overline{H}_n$ to a normal spanning tree of $\overline{H}_{n} \cup H_n$.
However, since $G$  is  uncountable and the subgraphs are just countable, one needs a transfinite sequence $H_1,H_2,H_3,\ldots, H_\omega, H_{\omega+1},\ldots$ of such subgraphs. This is where the gap in Diestel and Leader's proof occurs, when they advise in \cite[\S 5]{DiestelLeaderNST} to ``\emph{repeat\textnormal{[}...\textnormal{]} this step transfinitely until $T_F$ is exhausted}'', for this strategy may fail at limit steps. Indeed, even if one carefully constructs, as Diestel and Leader do, the first $\omega$ many graphs $H_1,H_2,H_3,\ldots$ such that each graph has only finitely many neighbours in the union of the earlier ones, there might be trouble finding a suitable $H_\omega$ that has only finitely many neighbours in $\overline{H}_\omega = \bigcup_{n < \omega} H_n$; indeed,  the current attempt is doomed at this point if there exists just one vertex $v \in G - \overline{H}_\omega$ with infinitely many neighbours in $\overline{H}_\omega$: to which $H_i$ for $i \geq \omega$ shall it belong?

Hence, implementing this strategy successfully requires a certain amount of ``looking ahead'' in order to avoid problems at limit steps. 
As case in point, consider again the $\omega_1$-regular tree with all tops $(T,\leq)$, and let $G_{\omega_1}$ be any $T$-graph. Select an arbitrary collection $F \subset T^{\omega}$ of tops, let $T(F)$ denote the subtree of $T$ induced by all finite sequences in $T$ together with $F$, and $G_{\omega_1}(F)$ be the corresponding subgraph of $G_{\omega_1}$. 

For $i < \omega_1$ write $G_i$  for the subgraph of $G_{\omega_1}(F)$ induced by all sequences in $T(F)$ with values strictly less than $i$. Implementing the strategy following Diestel and Leader, one could select for example $H_{n+1} = G_{n+1} \setminus G_n$, as is readily verified using Lemma~\ref{lem_Tgraphproperties}(\ref{itemT3}). However, any top $f \in F$ with $f(n) < \omega$ for all $n \in \N$ but $\sup \set{f(n)}:{n \in \N} = \omega$ is then precisely such a vertex in $G - \overline{H}_\omega$ with  infinitely many neighbours in $\overline{H}_\omega$ that we are trying to avoid. 

Formalizing this observation, for $f \in F$ define $f^* := \sup \set{f(n)}:{n \in \N} < \omega_1$, $F^* = \set{f^*}:{f \in F} $, and $F^*_> = \set{f^* \in F^*}:{f^* >f(n) \text{ for all } n \in \N} \subset  \omega_1$, the subset of $F^*$ where the supremum is proper. Using this notation, selecting $H_{n+1} = G_{c_{n+1}} \setminus G_{c_n}$ for some increasing sequence $c_1 < c_2 < \cdots$ avoids this problem at the first limit step precisely if $c_\omega := \sup \set{c_n}:{n \in \N}$ does not belong to $F^*_>$; and it avoids the problem altogether if $C = \set{c_i}:{i < \omega_1}$ is a club set of $\omega_1$ with $C \cap F^*_> = \emptyset$. Such a club set $C$ exists if and only if $F^*_>$ fails to be stationary. In other words, the strategy suggested by Diestel and Leader can be carried out precisely when there is a suitable club set $C$ along which we can decompose the graph:

\begin{theorem}
\label{thm_clubdecomp}
A graph of the form $G_{\omega_1}(F)$ has a normal spanning tree if and only if it contains no $(\aleph_0,\aleph_1)$-subgraph and some club set $C \subset \omega_1$ avoids $F^*_>$.
\end{theorem}

\begin{proof}
	We first prove the forwards implication. If $G_{\omega_1}(F)$ has a normal spanning tree, then it clearly cannot contain an $(\aleph_0,\aleph_1)$-subgraph. Now  assume for a contradiction that $F^*_>$ meets every club set of $\omega_1$. Then $S=F^*_>$ is stationary. For every $s \in S$ choose some $f_s \in F$ with $f_s^* = s$. Then $F(S) = \set{f_s}:{s \in S}$ gives rise to a subgraph of the form $G_{\omega_1}(S) \subset G_{\omega_1}(F)$ that fails to have a normal spanning tree by Lemma~\ref{lem_stationarytops}, a contradiction. 
	
	Conversely, assume that there is a club-set $C \subset \omega_1$ avoiding $F^*_>$. We show that $G_{\omega_1}(F)$ has a normal spanning tree unless it contains an $(\aleph_0,\aleph_1)$-subgraph. Without loss of generality, $G_{\omega_1}(F)$ is   the $T(F)$ graph where all comparable nodes are connected by an edge (if this graph has a normal spanning tree, then also all its connected subgraphs have normal spanning trees). This ensures that $\downcl{t}$ spans a clique for all $t \in T(F)$.
	
	Write $G_i$  for the subgraph of $G_{\omega_1}(F)$ induced by all sequences with values strictly less than $i$. If some  $G_i$ is uncountable, there must be uncountably many tops from $F$ above the countable subtree $G_i \cap T^{<\omega}$, giving rise to an $(\aleph_0,\aleph_1)$-subgraph. Hence, all $G_i$ are countable.
	
	Let $C= \set{c_i}:{i < \omega_1}$ be an increasing enumeration of the club set $C$. Then $G$ is the  increasing union over $\bigcup \set{G_{c_i}}:{i< \omega_1}$.  Moreover, this  union  is continuous precisely because $F^*_> \cap C = \emptyset$: an element $f \in G_{c_\ell} \setminus \bigcup_{i < \ell} G_{c_i}$ for a limit $\ell < \omega_1$ would satisfy $f^* = c_\ell$ and $f(n) < c_\ell$ for all $n \in \N$, and hence $f^* \in F^*_> \cap C$. 
	
	This allows us to construct -- by a transfinite recursion on $i < \omega_1$ -- an increasing chain of normal spanning trees $R_i$ of $G_{c_i}$ all with the same root extending each other. Assume that the normal spanning tree $R_i$ of $G_{c_i}$ is already defined. By Lemma~\ref{lem_Tgraphproperties}(\ref{itemT3}), the  components of $G_{c_{i+1}} \setminus G_{c_i}$ are spanned by the upsets $\upcl{t}$ for $t$ the $T$-minimal elements of $G_{c_{i+1}} \setminus G_{c_i}$. By definition of $G_{c_i}$, the down-closure $\downcl{t}$ for any such $t$ forms a finite clique in $G$. Hence, $\downcl{t}$ forms a chain in the normal spanning tree $R_i$. Let $r_t$ denote the maximal element of $\downcl{t}$ in $R_i$.
 Since $G_{c_{i+1}}$ is countable, there is by Theorem~\ref{thm_Jung}(2) for every component $\upcl{t}$ of $G_{c_{i+1}} \setminus G_{c_i}$ a normal spanning tree $R_t$ with root $t$. Attaching the $R_t$'s to $R_i$ with $t$ a successor of $r_t$ gives a normal spanning tree $R_{i+1}$ of $G_{c_{i+1}}$. 

By the continuity of our sequence, for any limit $\ell < \omega_1$, the union $R_\ell = \bigcup_{i < \ell} R_i$ is a normal spanning tree for $G_{c_\ell}$. Then $R = \bigcup_{i < \omega_1} R_{i}$ is the desired normal spanning tree for $G_{\omega_1}(F)$.
\end{proof}

\begin{cor}
\label{cor_clubdecomp}
A graph of the form $G_{\omega_1}(F)$ has a normal spanning tree if and only if it does not contain an $(\aleph_0,\aleph_1)$-subgraph or a subgraph isomorphic to $G_{\omega_1}(S)$ for $S \subset \omega_1$ stationary.
\end{cor}

\section{Irreducible obstructions of size \texorpdfstring{$\kappa$}{kappa}}
\label{sec_counterexamples2}

One particular consequence of Diestel and Leader's proposed forbidden minor characterisation would have been that `not having a normal spanning tree' is a property reflecting to at least one minor of size $\aleph_1$. 
However, it turns out that for all uncountable regular cardinals $\kappa$ there are $\kappa$-sized graphs without a normal spanning tree that consistently have the property that all their minors of size $<\kappa$ do have normal spanning trees.

These examples are natural generalisations of our earlier examples from Theorem~\ref{thm_counter} to arbitrary regular uncountable cardinals $\kappa$. They are defined as follows: Consider the $\kappa$-regular tree with all tops, represented by all sequences of elements of $\kappa$ of length $\leq\omega$, and let $G_{\kappa}$ be any $T$-graph. Given a set $S \subset \kappa$ of limit ordinals \emph{of countable cofinality}, choose for each $s \in S$ a cofinal sequence $f_s \colon \N \to s$, and let $F=F(S) = \set{f_s}:{s \in S}$ be the corresponding collection of sequences in $\kappa$. Let $T(S)$ denote the subtree of $T$ given by all finite sequences in $T$ together with $F(S)$, and let $G_{\kappa}(S)$ denote the corresponding induced subgraph of $G_{\kappa}$.

\begin{theorem}
\label{thm_aleph2obstruction}
Let $\kappa$ be a regular uncountable cardinal. Whenever $S \subset \kappa$ is stationary consisting just of cofinality $\omega$ ordinals, then $G_{\kappa}(S)$ does not have a normal spanning tree. 

Furthermore, it is consistent with the axioms of set theory ZFC that for any regular uncountable cardinal $\kappa$ there exists such a stationary set $S \subset \kappa$ such that all minors of $G_\kappa(S)$ of cardinality strictly less than $\kappa$ have normal spanning trees.
\end{theorem}

\begin{proof}
Indeed, that $G_{\kappa}(S)$ does not have a normal spanning tree follows as in Lemma~\ref{lem_stationarytops}.

To see the furthermore part of the theorem, recall that assuming Jensen's square principle $\square_{\kappa}$, which holds for example in the constructible universe, there exists a \emph{non-reflecting} stationary set in $\kappa$, i.e.\ a stationary set $S \subset \kappa$ consisting just of cofinality $\omega$ ordinals such that for any limit ordinal $\ell < \kappa$, the restriction $S \cap \ell$ is not stationary in $\ell$, see e.g. \cite[\S4]{cummings2005notes}.

It remains to show that for any non-reflecting stationary set $S \subset \kappa$ of cofinality $\omega$ ordinals, the graph $G_{\kappa}(S)$ is a graph without normal spanning tree such that all minors of size $< \kappa$ do have normal spanning trees. As in the proof of Theorem~\ref{thm_clubdecomp}, we may assume that $G_{\kappa}(S)$ is the $T$-graph where all comparable vertices of $T(S)$ are adjacent.

For $i < \kappa$, denote by $G_i$ once again the subgraph of $G_\kappa(S)$ induced by all sequences in $T(S)$ with values strictly less than $i$. Every minor of size $< \kappa$ of $G_{\kappa}(S)$ is a contraction of some ${<}\kappa$-sized subgraph of $G_{\kappa}(S)$, and hence by regularity of $\kappa$, a minor of $G_i$ for some $i  < \kappa$.
Hence, it suffices to show that every $G_i$ for $i < \kappa$ has a normal spanning tree. This will be done by induction on $i$, following the proof idea in Theorem~\ref{thm_clubdecomp}.

The base case for the induction is trivial. In the successor step, fix a normal spanning tree $R$ for $G_i$. By definition of $F(S)$, no component $D$ of $G_{i+1} \setminus G_i$ can contain an element from $F$. Hence, by Lemma~\ref{lem_Tgraphproperties}(\ref{itemT3}), the tree order of $(T,\leq)$ restricted to such a component prescribes a normal spanning tree $R_D$ for $D$. As $N(D)$ is finite and complete in $G_i$, its elements form a chain in $R$. Hence, attaching $R_D$ to $R$ as a further uptree behind the highest element of $N(D)$ in $R$ for every such component $D$ gives rise to a normal spanning tree for $G_{i+1}$. 

At a limit step $\ell < \kappa$, by choice of $S$ there is a club set $C \subset \ell$ which misses $S$. Let $C= \set{c_i}:{i < \cf(\ell)}$ be an increasing enumeration of $C$. From $S \cap C = \emptyset$ is follows that $\set{G_{c_i}}:{i < \cf(\ell)}$ is an increasing, continuous chain in $G_\ell$ such that every component of $G_{c_{i+1}} \setminus G_{c_i}$ has finite neighbourhood in $G_{c_i}$. As in Theorem~\ref{thm_clubdecomp}, this allows us to construct an increasing chain of normal spanning trees $R_i$ of $G_{c_i}$ extending each other all with the root of $T$ as their root.
 Then $R = \bigcup_{i < \cf(\ell)} R_{i}$ is a normal tree in $G_\ell$.

If $\cf(\ell) > \omega$, then $G_\ell = \bigcup \set{G_{c_i}}:{i < \cf(\ell)}$ and we are done. Otherwise, if $\cf(\ell) = \omega$, we could have $f_\ell \in G_\ell \setminus  \bigcup \set{G_{c_i}}:{i < \cf(\ell)}$ if $\ell \in S$. In this case, we attach $f_\ell$ \emph{below} the root of $R$ to form a normal spanning tree of  $G_\ell$ rooted in $f_\ell$.
\end{proof}

It might come as a surprise to hear that no such examples as in the furthermore-part of Theorem~\ref{thm_aleph2obstruction} are possible at singular uncountable cardinals $\kappa$. Indeed, the property of having a normal spanning tree exhibits the following singular compactness-type behaviour: If a connected graph $G$ of singular uncountable size $\kappa$ has the property that all its subgraphs of size $<\kappa$ have normal spanning trees, then so does $G$ itself. For details, we refer the reader to \cite{pitz2020d}.

\section{Further problems on normal spanning trees and forbidden minors}

\begin{prob}
	Is there an Aronszajn tree graph that contains neither an $(\aleph_0,\aleph_1)$-graph nor a graph $G_{\omega_1}(S)$ as in Theorem~\ref{thm_counter} as a minor?
\end{prob}

\begin{prob}
	Does every $\aleph_1$-sized graph without normal spanning tree contain either an $(\aleph_0,\aleph_1)$-graph, an Aronszajn tree graph or a graph $G_{\omega_1}(S)$ as in Theorem~\ref{thm_counter} as a minor?
\end{prob}

\begin{prob}
Is it consistent with the axioms of set theory ZFC that every graph without normal spanning tree contains an $\aleph_1$-sized subgraph or minor without normal spanning tree?  
\end{prob}

\bibliographystyle{plain}
\bibliography{reference}

\end{document}